\documentclass[10pt]{amsart}

\usepackage{amssymb, dsfont} 
\usepackage{fullpage, graphicx} 
\usepackage{hyperref} 
\usepackage{enumitem} 
\usepackage{caption} 
\usepackage{xcolor} 
\usepackage{thmtools} 
\usepackage{mathrsfs} 

\hypersetup{colorlinks=true, urlcolor=purple, citecolor=blue, linkcolor=red, pdfstartview={XYZ null null 0.90}}

\DeclareMathOperator{\link}{link}
\DeclareMathOperator{\mat}{mat}
\DeclareMathOperator{\Int}{Int}
\DeclareMathOperator{\Riv}{Riv}

\DeclareMathOperator{\PSL}{PSL}

\newcommand{\ZZ}{\ensuremath{\mathbb{Z}}}
\newcommand{\RR}{\ensuremath{\mathbb{R}}}

\newcommand{\uhp}{\ensuremath{\mathbb{H}}} 

\newcommand{\sm}[4]{\ensuremath{\left(\begin{smallmatrix} #1 & #2\\#3 & #4\end{smallmatrix}\right)}}
\newcommand{\lm}[4]{\ensuremath{\left(\begin{matrix} #1 & #2\\#3 & #4\end{matrix}\right)}}
\newcommand{\loren}[1]{\ensuremath{\mathscr{L}(#1)}}
\newcommand{\IntRS}{\ensuremath{\Int^{\text{RS}}}}
\newcommand{\genmtx}{\sm{a}{b}{c}{d}}
\newcommand{\tl}[1]{\tilde{\ell}_{#1}}
\newcommand{\loseq}{\ensuremath{\mathcal{S}}}

\newtheorem{theorem}{Theorem}[section]

\newtheorem{lemma}[theorem]{Lemma}
\newtheorem{proposition}[theorem]{Proposition}

\theoremstyle{definition}

\newtheorem{definition}[theorem]{Definition}

\newtheorem{remark}[theorem]{Remark}

\numberwithin{equation}{section}

\begin{document}

\title[Linking number]{Linking number of modular knots}
\author[J. Rickards]{James Rickards}
\address{University of Colorado Boulder, Boulder, Colorado, USA}
\email{james.rickards@colorado.edu}
\urladdr{https://math.colorado.edu/~jari2770/}
\date{\today}
\thanks{I thank Vishal P. Patil; a conversation about knot theory with him inspired me to have a look at this problem again. This work was partially supported by NSF-CAREER CNS-1652238 (PI Katherine E. Stange).}
\begin{abstract}
We compute the linking number of two modular knots in the space $\text{PSL}(2, \mathbb{Z})\backslash\text{PSL}(2, \mathbb{R})$ with the trefoil filled in, which answers a question posed by Ghys in 2007. This computation is realized through the correspondence between modular links and Lorenz links, and can be thought of as an intersection number involving Conway topographs. We compare this to a second formula for the linking number of Lorenz links, which was proven by Stephen F. Kennedy in 1994.
\end{abstract}

\maketitle

\setcounter{section}{-1}
\section{Author's note}
After posting to Arxiv, the works of Christopher-Lloyd Simon have been brought to my attention. In his thesis \cite{Simonthe} and subsequent paper \cite{Simon22}, he studies properties of modular knots. A formula equivalent to Theorem \ref{thm:mainthm} is derived, and the connection to periodic paths on a trivalent tree (i.e. rivers on a topograph) are established. These results are his starting point for many more results relating to modular knots. I will leave this preprint up as an alternate exposition, but precedence will go to his work, and anyone reading this paper should consult and cite his work instead.

\section{Introduction}
The manifold $X=\PSL(2, \ZZ)\backslash\PSL(2, \RR)$ is diffeomorphic to $S^3$ minus a trefoil (for a proof, see \cite{Milnor71}). The modular flow on $X$ is given by right multiplication by $\phi(t)=\sm{e^t}{0}{0}{e^{-t}}$, and periodic orbits of the modular flow are called modular knots. These knots are parametrized by conjugacy classes of primitive hyperbolic matrices, as described in the following definition.

\begin{definition}
	Let $A\in\PSL(2, \ZZ)$ be a primitive hyperbolic matrix with largest eigenvalue $\lambda>1$. Then there exists a matrix $M\in\PSL(2, \RR)$ with $M^{-1}AM=\sm{\lambda}{0}{0}{\lambda^{-1}}$. Define the knot $k_A$ to be:
	\begin{align*}
		k_A:[0,\log(\lambda)) & \rightarrow X\\
		t & \rightarrow M\phi(t).
	\end{align*}
\end{definition}

The path $k_A$ is a knot since $M\phi(\log(\lambda))=AM\sim M=M\phi(0)$ in $X$. Furthermore, it does not depend on the choice of $M$, and is constant across the $\PSL(2, \ZZ)$ conjugacy class of $A$.

As part of his 2006 ICM address, Ghys studied these knots (\cite{Ghys07} and \cite{GhysLeys}). He proved that the linking number of $k_A$ with the removed trefoil can be given by $\mathfrak{R}(A)$, the Rademacher function of $A$ (see the paper by Atiyah, \cite{At87}, for various equivalent definitions). He then asked: what is the linking number of $k_A$ and $k_B$? 

One difficulty of this question is the presence of the missing trefoil: to compute the linking number of our knots, you need to ``fill the trefoil in''. However, if you don't leave the space $\PSL(2, \ZZ)\backslash\PSL(2, \RR)$, there is no obvious way in which to do that. One partial fix is due to Duke, Imamo\={g}lu, and T\'{o}th in 2017 (\cite{DIT17}). In their paper, they used weight 2 cocycles to generalize the Rademacher function, and proved that this generalization gives the linking number of $k_A+k_{A^{-1}}$ with $k_B+k_{B^{-1}}$. The link $k_A+k_{A^{-1}}$ is null-homologous, which enabled them to do their computations without leaving $\PSL(2, \ZZ)\backslash\PSL(2, \RR)$. This generalized Rademacher function was later studied by Matsusaka in \cite{Matsusaka20}.

As part of their paper, they noted that $\link(k_A+k_{A^{-1}},k_B+k_{B^{-1}})$ could be interpreted as an intersection number of closed geodesics on the modular curve. This topic was studied further in \cite{JR21}, and then generalized to the case of geodesics on a Shimura curve in \cite{JR21shim}. Since these results all deal with the links $k_A+k_{A^{-1}}$, they do not provide a full solution to Ghys's original question. By harnessing the connection to Lorenz links, in Theorem \ref{thm:mainthm} we give a combinatorial computation of $\link(k_A, k_B)$, settling the question. In Section \ref{sec:DIT}, we compare this finer invariant to the symmetrized version of Duke et al., and give some data to illustrate the difference.

We should note that the linking number of Lorenz knots has been computed before. In \cite{Kennedy94}, Stephen F. Kennedy gives formulas for the linking numbers of Lorenz knots, as well as horseshoe knots. This formula involves the alphabetization of words, and the fact that the two formulas coincide does not appear to be trivial. For sake of comparison, we record his main result in Section \ref{sec:kennedy}.

\section{Main result}

This paper will take us into the world of dynamical systems and Lorenz equations, so we will adapt their terminology.

\begin{definition}
	A \textit{Lorenz word} is any finite aperiodic word in the letters $L, R$, with length at least 2. A \textit{single shift} of a Lorenz word $W$ is the word obtained by taking the first letter and moving it to the end, and is denoted $s(W)$. A \textit{cyclic shift} is the result of any number of single shifts. Two Lorenz words are said to be equivalent if they differ by a cyclic shift. A \textit{Lorenz sequence} is a doubly infinite periodic sequence, formed by repeating a fixed Lorenz word in both directions. Denote the Lorenz sequence associated to a Lorenz word $W$ by $\loseq(W)$.
\end{definition}

Given a Lorenz word $W$, we can substitute in the following matrices for $L$ and $R$ and multiply out to produce a primitive hyperbolic matrix:
\[L:=\lm{1}{1}{0}{1},\qquad R:=\lm{1}{0}{1}{1}.\]
Denote the matrix formed by $\mat(W)$. Given a primitive hyperbolic matrix $A$, define $\loren{A}$ to be the set of Lorenz words $W$ for which $A$ is conjugate to $\mat(W)$ over $\PSL(2, \ZZ)$. The following lemma is classical (and can be easily proven with the material in Section \ref{sec:topograph}).

\begin{lemma}\label{lem:loren}
	The set $\loren{A}$ is non-empty, and consists of a single equivalence class of Lorenz words.
\end{lemma}

The set $\loren{A}$ is clearly constant across an conjugacy class of primitive hyperbolic matrices. Since $A$ is hyperbolic, the equation $Ax=x$ has two real solutions, called the \textit{roots} of $A$ (where $A$ acts by M\"{o}bius transformation). As we will describe in Remark \ref{rem:contfrac}, $\loren{A}$ can be computed directly from the continued fraction representation of one of the roots of $A$.

\begin{theorem}\label{thm:mainthm}
	Let $A,B\in\PSL(2, \ZZ)$ be primitive hyperbolic matrices that are not conjugate to each other. Let $W_A\in\loren{A}$ and $W_B\in\loren{B}$, and write $W_A=a_1a_2\cdots a_m$ and $W_B=b_1b_2\cdots b_n$. Then $-\link(k_A, k_B)$ is equal to the number of triples of integers $(i, j, x)$ such that:
	\begin{itemize}
		\item $1\leq i\leq m$, $1\leq j\leq n$, and $x\geq 0$;
		\item $a_i=L$ and $b_j=R$;
		\item $a_{i+k}=b_{j+k}$ for all integers $1\leq k\leq x-1$;
		\item $a_{i+x}=R$ and $b_{j+x}=L$;
	\end{itemize}
	where the indices are taken modulo the periods ($m$ and $n$ respective to $a$ and $b$). In particular, $\link(k_A, k_B)$ is always a negative integer.
\end{theorem}

Another interpretation of the combinatorial computation in Theorem \ref{thm:mainthm} is we are computing (modulo the periods) occurrences of (possibly empty) words $W'$ such that $LW'R$ appears in $\loseq(W_A)$ and $RW'L$ appears in $\loseq(W_B)$.

For an example of Theorem \ref{thm:mainthm}, consider the matrices $A=\sm{3}{5}{7}{12}$ and $B=\sm{2}{1}{1}{1}$, which correspond to words $W_A=RRLLRL$ and $W_B=LR$. In Figure \ref{fig:twoknots}, we depict the knots $k_A$ and $k_B$ in $\RR^3$ with the trefoil filled in, and can numerically compute that their linking number is $-3$ (using the standard convention that an overcrossing from left to right has sign $+1$).

\begin{figure}[htb]
	\includegraphics[scale=1]{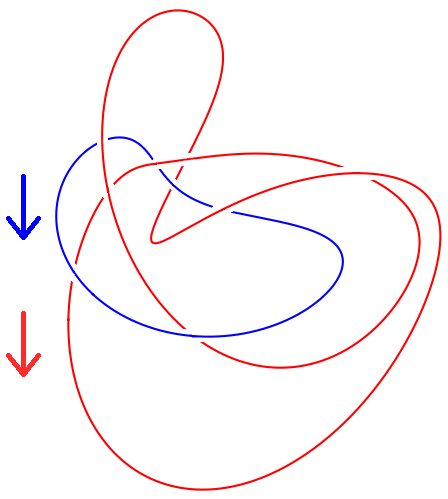}
	\caption{$k_A$ is in red and $k_B$ is in blue. Image created with Matplotlib, \cite{matplotlib}.}\label{fig:twoknots}
\end{figure}

In view of Theorem \ref{thm:mainthm}, we pick up the three triples $(i, j, x)=(3,2,4),(4,2,0),(6,2,0)$, corresponding to

\begin{minipage}{.31\textwidth}
	\begin{align*}
		\ldots\mathbf{RR} & \boxed{\mathbf{LLRL}RR}LLRL\ldots\\
		\ldots\mathbf{L}  & \boxed{\mathbf{R}LRLRL}R\ldots\\
	\end{align*}
\end{minipage}%
\begin{minipage}{.31\textwidth}
	\begin{align*}
		\ldots\mathbf{RRL} & \boxed{\mathbf{LR}}\mathbf{L}RRLLRL\ldots\\
		\ldots\mathbf{L} & \boxed{\mathbf{R}L}R\ldots\\
	\end{align*}
\end{minipage}%
\begin{minipage}{.31\textwidth}
	\begin{align*}
		\ldots\mathbf{RRLLR} & \boxed{\mathbf{L}R}RLLRL\ldots\\
		\ldots\mathbf{L} & \boxed{\mathbf{R}L}R\ldots\\
	\end{align*}
\end{minipage}%

The strategy to prove Theorem \ref{thm:mainthm} is to use Ghys's result that modular links are isotopic to Lorenz links (top of page 272 of \cite{Ghys07}). We then compute the linking number of Lorenz links corresponding to the given Lorenz words using Birman–Williams’ template theory, \cite{BW83}.

\begin{remark}
	Theorem 4.1 of \cite{BW83} proves that the linking number of Lorenz links in non-zero, which is a corollary of Theorem \ref{thm:mainthm}.
\end{remark}
\begin{remark}
	The work of Birman-Williams (along with most other Lorenz knot papers) uses the opposite sign convention for the linking number, so that it is always positive. Since we framed the question in terms of the number theoretic picture, we will instead follow the convention used by Ghys and Duke-Imamo\={g}lu-T\'{o}th in \cite{Ghys07} and \cite{DIT17}.
\end{remark}

\section{Connection to Conway's topograph and quadratic forms}\label{sec:topograph}

Despite the strong connection to Lorenz links, motivation for the formula of Theorem \ref{thm:mainthm} came from the aforementioned works \cite{DIT17} and \cite{JR21}. Given a primitive hyperbolic $A\in\PSL(2, \ZZ)$, denote the geodesic connecting the two roots of $A$ by $\ell_{A}$. This descends to a closed geodesic $\tilde{\ell}_A$ on the modular curve, $\PSL(2, \ZZ)\backslash\uhp$. The unsigned intersection number of two closed geodesics $\ell_1$ and $\ell_2$ , denoted $\Int(\ell_1,\ell_2)$, counts the number of transverse intersections of the (unoriented) curves.

\begin{theorem}[Follows from Theorem 6.3 and Lemma 6.7 of \cite{DIT17}]\label{thm:dit}
	Let $A,B\in\PSL(2, \ZZ)$ be primitive, hyperbolic, and not conjugate to each other or each other's inverse. Then
	\[\link(k_A+k_{A^{-1}}, k_B+k_{B^{-1}})=-\Int(\tl{A}, \tl{B}).\]	
\end{theorem}

While it's not obvious how this intersection number can break up into a sum of four linking numbers, a natural decomposition appears when we consider the Conway topographs of the corresponding quadratic forms.

Given a primitive hyperbolic matrix $A=\genmtx\in\PSL(2, \ZZ)$, the equation $Ax=x$ translates to $cx^2+(d-a)x-b=0$. Let $g=\gcd(c, d-a, b)$, and the quadratic form associated to $A$ is
\[q_A(x, y):=\dfrac{1}{g}\left(cx^2+(d-a)xy-by^2\right):=\dfrac{1}{g}[c, d-a, -b].\]
This is a primitive integral indefinite binary quadratic form, and conjugacy classes of matrices correspond to $\PSL(2,\ZZ)$-equivalence classes of quadratic forms. Furthermore, this is a bijection, with the inverse map realized by taking the automorph of $q_A$ (for more details, see Proposition 1.4 of \cite{Sar82}).

The Conway topograph is a combinatorial object associated to a quadratic form (for a more comprehensive study of the topograph, see Chapter 4 of \cite{Hat22}, or Section 3 of \cite{JR21}). The base object of the topograph is an infinite connected 3-regular graph embedded in $\RR^2$. In particular, if we add an orientation to an edge, there is a well-defined notion of ``left'' and ``right''. A path can therefore be represented by a word in $L$ and $R$, denoting left and right respectively.

The topograph for $q(x, y)$ divides $\RR^2$ into regions, and numbers can be placed in the regions and on the edges such that:
\begin{itemize}
	\item Numbers in regions represent values properly represented by $q(x, y)$ (i.e. $q(x, y)$ when $x,y\in\ZZ$ are coprime);
	\item If an edge contains the number $b$ and is adjacent to regions with numbers $a,c$, then $[a, \pm b, c]$ is a quadratic form similar to $q$. In fact, the entire equivalence class of forms similar to $q$ arises in this fashion.
	\item By assigning a ``positive direction'' to each edge, we can determine if we need to take $+b$ or $-b$ in the form. 
\end{itemize}
See Figure \ref{fig:top1} for part of the topograph of $q(x, y)=7x^2+9xy-5y^2$.

\begin{figure}[htb]
	\includegraphics[scale=1]{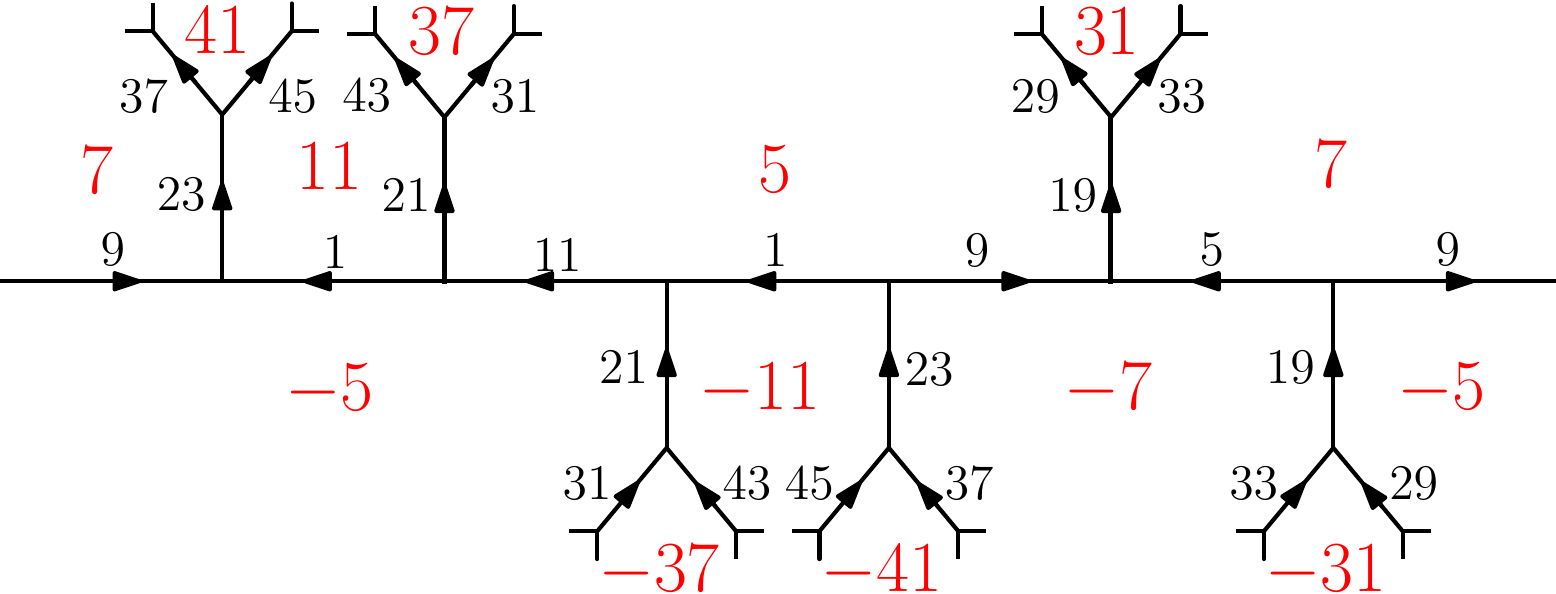}
	\caption{Topograph of $[7, 9, -5]$. Red numbers are in regions, and black numbers are on edges. The direction of the arrow determines the sign of the edge numbers.}\label{fig:top1}
\end{figure}

When $q$ is indefinite, there is a doubly infinite path (called the \textit{river}) that separates the positive and negative numbers placed in regions. The river is periodic with minimal period of length at least 2, and therefore can be represented as a Lorenz sequence.

\begin{definition}
	Let $q(x, y)$ be a primitive integral indefinite binary quadratic form. A \textit{river word} for $q$ is a Lorenz word that represents the minimal period of the path taken by the river on the topograph of $q$. It has length at least 2, and is unique up to cyclic shift. Let $\Riv(q)$ denote the set of river words of $q$
\end{definition}

\begin{remark}[Remark 3.5 of \cite{JR21}]\label{rem:contfrac}
Let $q=[a, b, c]$ have discriminant $D$, and let the continued fraction of $\frac{-b+\sqrt{D}}{2a}$ (the first root of $q$) be 
\[[a_0,a_1,\ldots]=[a_0,a_1,\ldots,a_s,\overline{a_{s+1},\ldots,a_{s+p}}],\]
where $s$ is the smallest integer such that the continued fraction is periodic after index $s$, and $p$ is the smallest \textit{even} integer such that the sequence has period $p$. Consider the sequence of $0$'s and $1$'s formed by:
\begin{itemize}
	\item $s+1\pmod{2}$ repeated $a_{s+1}$ times;
	\item $s+2\pmod{2}$ repeated $a_{s+2}$ times;
	\item $\cdots$
	\item $s+p\pmod{2}$ repeated $a_{s+p}$ times.
\end{itemize} 
If we replace the $0$s with $L$s and $1$s with $R$s, then the word formed is in $\Riv(q)$. For example, the first root of $[7, 9, -5]$ is $\frac{-9+\sqrt{221}}{14}$, which has continued fraction $[0,\overline{2,2,1,1}]$. This gives the river word $RRLLRL$, which agrees with Figure \ref{fig:top1}: start on the left hand side at $[7,9,-5]$, walk via $RRLLRL$, and you end up at another river edge representing $[7,9,-5]$.
\end{remark}

When drawing the topograph of an indefinite form, the river is normally ``flattened'' with the positive numbers above the river.

\begin{remark}
	Let $q=[a, b, c]$ be a ``positive river form'', i.e. $a>0$ and $c<0$, and let $W$ be the river word formed by starting at $q$. Then $\mat(W)$ is the automorph of $q$, i.e. $\mat(W)\circ q=q$.
\end{remark}

The topograph machinery can also be used to give a nice proof of Lemma \ref{lem:loren}. If $A$ is a primitive hyperbolic matrix, then $\Riv(q_A)=\loren{A}$.

To connect this back to Theorem \ref{thm:dit}, let $A,B\in\PSL(2,\ZZ)$ be primitive and hyperbolic. Let $T_A$ be the topograph associated to $q_A$, and let $T_B$ be the topograph associated to $q_B$. By picking an oriented edge of each topograph, we can superimpose one on top of the other, so that a region has an associated ordered pair: the number from $T_A$, followed by the number from $T_B$. There are four possibilities for the signs of the numbers in this pair: $(+,+),(+,-),(-,+),(-,-)$, and it turns out that we get an intersection between $\tl{A}$ and $\tl{B}$ if all four combinations appear (see \cite{JR21} for more details). This is equivalent to the rivers of the topgraphs starting off disjoint, meeting, and then crossing each other.

In particular, if we have ``intersecting topographs'', then you can translate either topograph by the corresponding river word to get another pair of intersecting topographs. This equivalence gives rise to the same intersection point on the modular curve, and is the only way to produce the same point. Theorem \ref{thm:inttop} follows.

\begin{theorem}[Theorem 5 of \cite{JR21}]\label{thm:inttop}
	Let $A,B\in\PSL(2, \ZZ)$ be primitive and hyperbolic, corresponding to respective topographs $T_A$ and $T_B$. Then $\Int(\tl{A}, \tl{B})$ is equal to the number of ways to superimpose $T_B$ on top of $T_A$ so that the rivers meet and cross, modulo the periods of the rivers.
\end{theorem}

Consider the combinatorics of Theorem \ref{thm:inttop}: start with $T_A$, flattened so that the river $R_A$ is horizontal. In order to meet and cross, the river $R_B$ for $T_B$ can meet $R_A$ from either the left or the right hand side (i.e. top/bottom in terms of the picture). It can then flow in the same or the opposite direction, leading to 4 possibilities.

\begin{definition}
	Let $\IntRS(A, B)$ denote the topograph intersection number where $R_B$ joins $R_A$ from the right hand side, and the rivers flow in the same direction. Call this the RS-intersection number.
\end{definition}

The RS-intersection number is exactly the quantity we want.

\begin{theorem}\label{thm:topint}
	The topograph RS-intersection number, $\IntRS(A, B)$, coincides with $-\link(k_A, k_B)$.
\end{theorem}

Before giving the proof, consider the example from Figure \ref{fig:twoknots}, i.e. $A=\sm{3}{5}{7}{12}$ and $B=\sm{2}{1}{1}{1}$, corresponding to river words $RRLLRL$ and $LR$. The topograph for $A$ was displayed in Figure \ref{fig:top1}, and all 3 possible RS-intersections are shown in Figure \ref{fig:topint}.

\begin{figure}[htb]
	\includegraphics[scale=0.85]{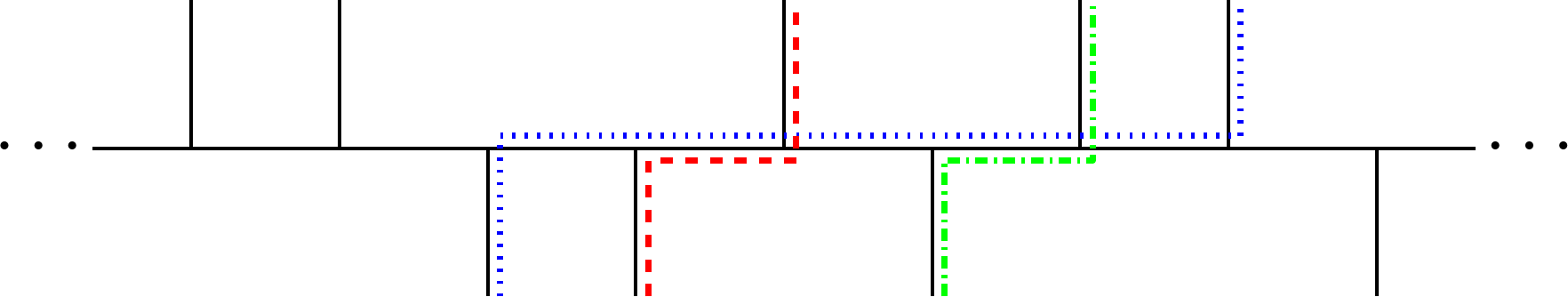}
	\caption{RS-intersection of $RRLLRL$ and $LR$. One and a half periods of the topograph of $RRLLRL$ are shown in solid black, and the three RS-intersections with $LR$ are in blue/red/green and dotted/dashed.}\label{fig:topint}
\end{figure}

\begin{proof}[Proof of Theorem \ref{thm:topint}]
Consider how to combinatorially compute $\IntRS$. Let $R_A$ correspond to the word $a_1a_2\cdots a_m$, and let $R_B$ correspond to the word $b_1b_2\cdots b_n$. For $R_B$ to join $R_A$ from the right hand side, we must have a pair of indices $(i, j)$ with $1\leq i\leq m$ and $1\leq j\leq n$ such that $a_i=L$ and $b_j=R$. Then the rivers flow in the same direction for $x\in\ZZ^{\geq 0}$ steps, corresponding to $a_{i+k}=b_{j+k}$ for $1\leq k\leq x-1$. Since $R_B$ crosses $R_A$, we must exit with $a_{i+x}=R$ and $b_{j+x}=L$. In particular, this is exactly the same description as provided in Theorem \ref{thm:mainthm}, so it follows from the proof in Section \ref{sec:proof}.
\end{proof}

\begin{remark}
	On the topograph side, $\IntRS(A, B)$ is still defined when $A$ and $B$ are conjugate. On the modular curve side of things, this corresponds to transverse self-intersections of the geodesic. However, it is not entirely clear what the interpretation should be in terms of knot theory. If $A=\sm{2}{1}{1}{1}$, then $\IntRS(A, A)=1$, whereas $k_A$ is the unknot. Perhaps there is a natural framing of the modular knots so that $\IntRS(A, A)$ becomes the self-linking number.
\end{remark}

\section{Lorenz knots}

In an attempt to model atmospheric convection, meteorologist Edward Norton Lorenz came up with the following system of differential equations (where $t$ represents time):
\[\dfrac{dx}{dt}=10(y-x),\quad \dfrac{dy}{dy}=28x-y-xz,\quad \dfrac{dz}{dt}=xy-\dfrac{8}{3}z.\]
By picking an initial starting point, the ODE follows a path to determine a flow. See Figure \ref{fig:lorenz} for an example.

\begin{figure}[htb]
	\includegraphics[scale=0.6]{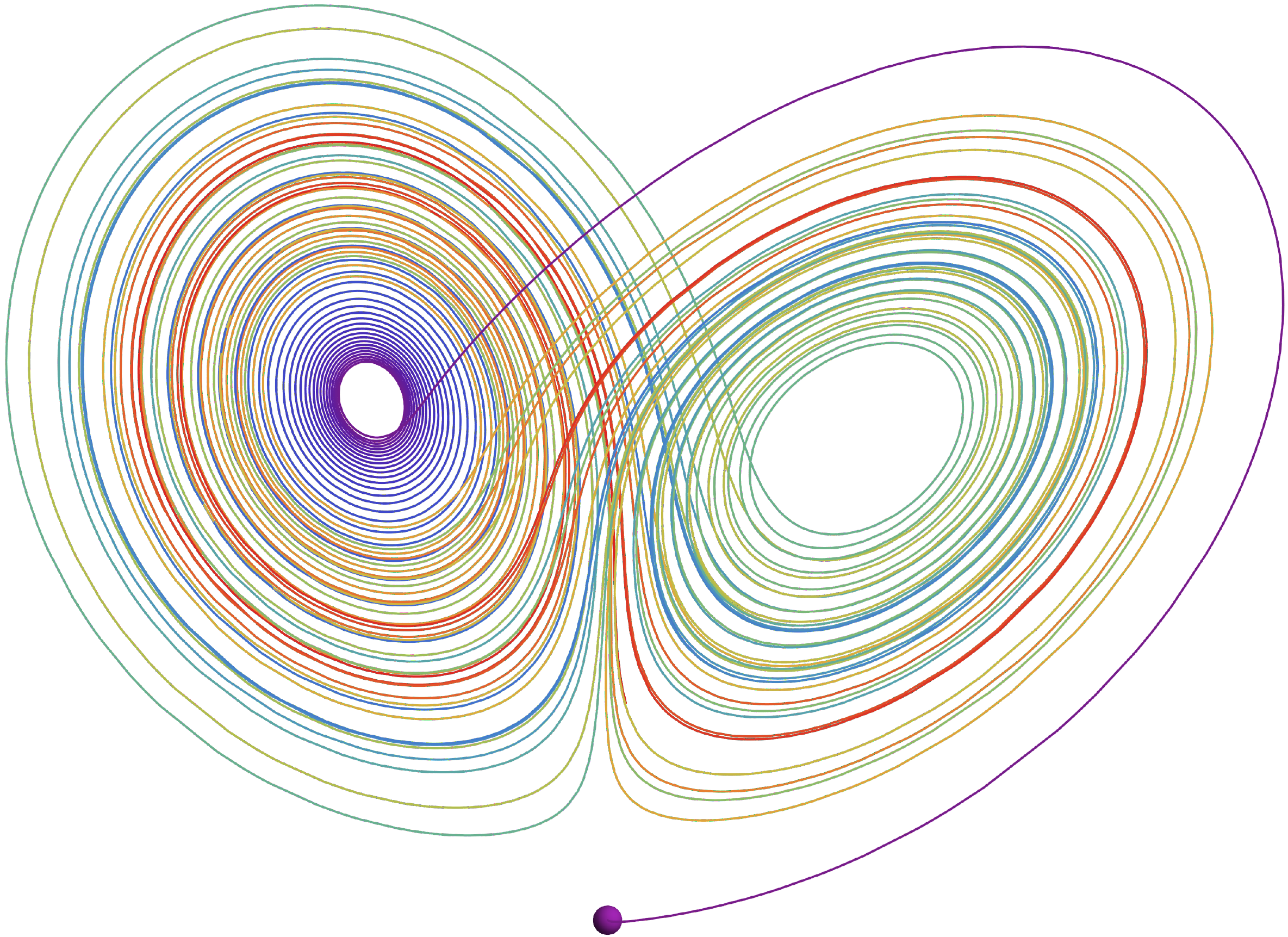}
	\caption{A sample solution to the Lorenz system. Image created with Mathematica, \cite{mathematica}.}\label{fig:lorenz}
\end{figure}

In \cite{Lorenz63}, Lorenz proved that while paths must all eventually enter and remain in a bounded region, they are very susceptible to miniscule changes in the initial inputs. This observation led to the beginning of chaos theory.

We will study a different aspect of this theory, namely the knots formed as solutions to Lorenz's equations. A \textit{Lorenz knot} is defined to be a closed periodic orbit of this ODE, and a \textit{Lorenz link} is a set of Lorenz knots.

Lorenz knots/links can be studied with template theory, introduced by Birman and Williams in \cite{BW83}. As seen in Figure \ref{fig:lorenz}, the solutions seem to form a ``butterfly'': looping around one of two circles. This can be made precise with the \textit{Lorenz template}, which is a branched surface in $\RR^3$ with a semi-flow that describes the behaviour of Lorenz knots. See Figure \ref{fig:template} for a depiction of the Lorenz template.

\begin{figure}[h]
	\includegraphics[scale=1]{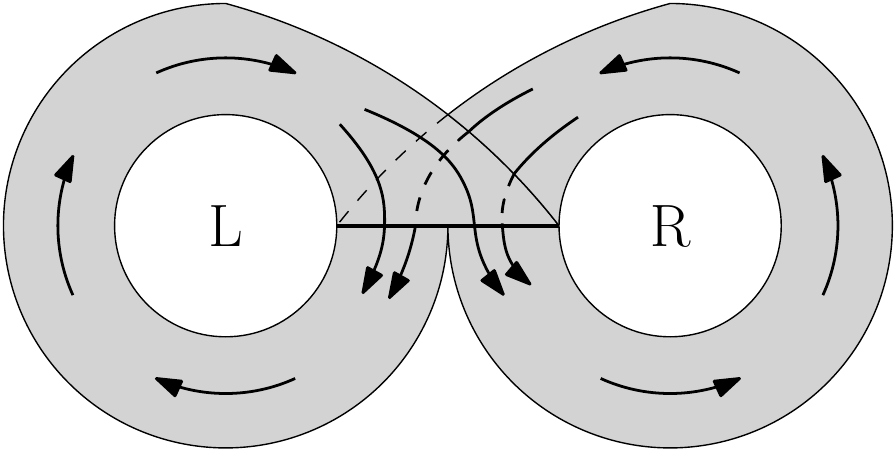}
	\caption{The Lorenz template.}\label{fig:template}
\end{figure}

\begin{remark}
	The basic tools to create the Lorenz template came from the work of Guckenheimer and Williams in \cite{GuckWill79}, \cite{Will79}. The proof that the template actually corresponds to Lorenz's original equations is due to Tucker in \cite{Tucker02}.
\end{remark}

Let the branch (the piece connecting the left and right circles) be $[0, 1]$. A path starting at $x\in(0, 1)-\{1/2\}$ will wind around the left (if $x<1/2$) or the right (if $x>1/2$) loop, ending back up at $f(x)\in(0, 1)$. Furthermore, we have the following properties:
\begin{itemize}
	\item $f(x)\neq x$ for all $x\in(0, 1)-\{1/2\}$;
	\item $f$ restricted to $(0, 1/2)$ is a continuous increasing bijection with $(0, 1)$;
	\item $f$ restricted to $(1/2, 1)$ is a continuous increasing bijection with $(0, 1)$;
	\item flows corresponding to $x$ and $x'$ with $0<x<x'<1/2$ (or $1/2<x<x'<1$) do not cross;
	\item flows coming from the left loop meet the branch above flows coming from the right loop (denoted by dotted lines in the figure).
\end{itemize}

In fact (see Section 2.4 of \cite{BW83}), the function $f(x)=2x\pmod{1}$ will model the Lorenz template.

\begin{remark}
	In Figure \ref{fig:lorenz}, the right hand side of the orbits come back to the branch in front of the left hand side orbits. This is opposite to what is done in the Lorenz template, Figure \ref{fig:template}. I'm not certain of why this is happens, but it is the convention used in \cite{BW83} and \cite{Ghys07}, so we will follow it. In any case, the only side affect is a potential difference in sign.
\end{remark}

Since a Lorenz knot is a periodic closed flow, we can associate a word to it via the sequence of $L$'s and $R$'s it follows on the template.

\begin{definition}
	The \textit{Lorenz word} of a Lorenz knot is the sequence of $L$'s and $R$'s that the knot visits as it travels the template over one period.
\end{definition}

It is clear that the Lorenz word is only defined up to cyclic shifts. Note the similarity between the Lorenz word of a Lorenz knot and the Lorenz word of a primitive hyperbolic matrix! By the remarkable work of Ghys in \cite{Ghys07}, these two worlds are the same.

\begin{theorem}[Ghys, 2006]\label{thm:ghys}
	Consider a modular link corresponding to distinct Lorenz words $W_1, W_2, \ldots, W_n$. This coincides (knot-theoretically) with the Lorenz link corresponding to Lorenz words $W_1, W_2, \ldots, W_n$.
\end{theorem}

In particular, to prove Theorem \ref{thm:mainthm}, it suffices to prove it for Lorenz links.

\section{Proof of the main result}\label{sec:proof}

To compute $\link(k_A, k_B)$, by Theorem \ref{thm:ghys}, it suffices to compute the linking number of the corresponding knots on the Lorenz template. Let these knots be $k_1'$ and $k_2'$ respectively. The linking number is then the number of times $k_2'$ crosses under $k_1'$, and accounting for sign: $+1$ from right to left, and $-1$ from left to right.

Let $W_1=a_1a_2\cdots a_m$ and $W_2=b_1b_2\cdots b_n$ be the distinct Lorenz words corresponding to $A$ and $B$ respectively. For each $1\leq i\leq m$, let $K_i$ be the part of the $k_1'$ corresponding to $a_i$ on the template, and say that it leaves from $x_i\in (0, 1)$ (and therefore flows to $x_{i+1}$, with indices taken modulo $m$). Similarly, for each $1\leq j\leq n$, let $L_j$ be the part of $k_2'$ on the template corresponding to $b_j$, and say that it leaves from $y_j\in(0, 1)$. To compute $\link(W_1, W_2)$, it suffices to compute $\link(K_i, L_j)$ for each pair $(i,j)$ with $1\leq i\leq m$ and $1\leq j\leq n$ and add them up.

\textbf{Case 1:} $a_i=b_j$. In this case, both knots flow around the same side of the template, not intersecting each other, and return to the branch in the same order as they started. See Figure \ref{fig:sameside} for a depiction of the situation. This contributes nothing to the linking number.

\begin{figure}[htb]
	\includegraphics[scale=0.85]{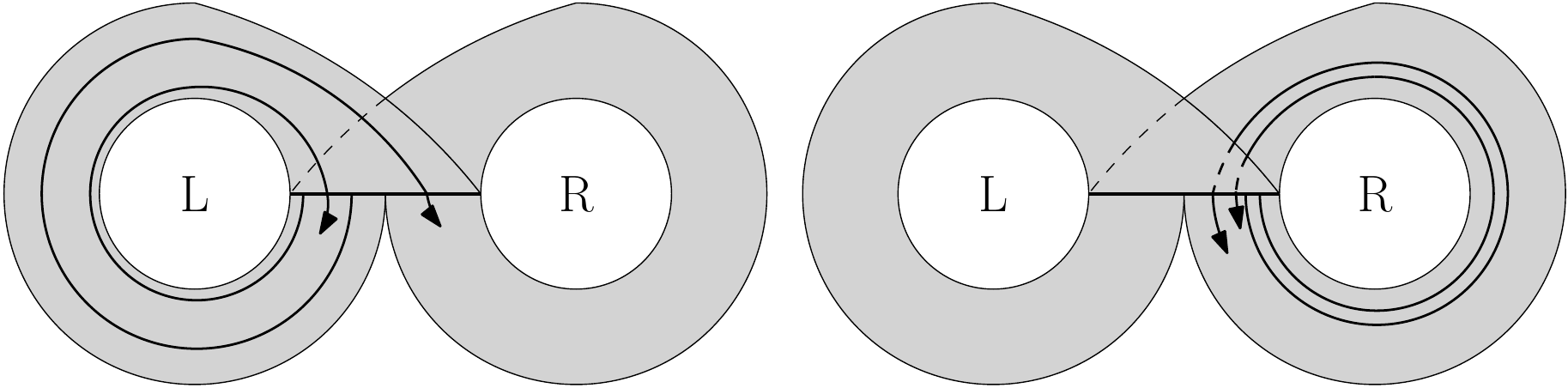}
	\caption{Flowing around the same side of the template.}\label{fig:sameside}
\end{figure}

\textbf{Case 2:} $a_i=L$ and $b_j=R$. The knots flow around opposite loops, and may or may not cross as they return. 

\begin{itemize}
	\item If they do not cross, then we must have $x_{i+1}<y_{i+1}$. Therefore we must have either $a_{i+1}=b_{j+1}$, or $a_{i+1}=L$ and $b_{j+1}=R$. In the first case, the loops will flow around the corresponding side, and again retain their ordering. Therefore we again have $a_{i+2}=b_{j+2}$ or $a_{i+2}=L$ and $b_{j+2}=R$. This process continues, until eventually the knots take opposite sides (after at most $mn$ iterations), which must be $L$ for $a_{i+k}$ and $R$ for $b_{j+k}$.
	
	\noindent To sum up this sub-case, if they do not cross, then there is some word $X$ such that the Lorenz words are $LXL$ and $RXR$ starting at $(a_i, b_j)$.
	
	\item If they do cross, then we have $y_{i+1}<x_{i+1}$. The sign of crossing is also $-1$, since the left branch comes in on top of the right branch. The analysis of the previous case holds, except now when the knots separate, we must have $a_{i+k}=R$ and $b_{i+k}=L$. In other words, there is some $X$ such that the Lorenz words are $LXR$ and $RXL$ starting at $(a_i, b_j$).
\end{itemize}

Since the conditions found in the two sub-cases are disjoint and cover every possibility, they are if and only if. In particular, occurrences of $LXR$ and $RXL$ contribute $-1$ to the linking.

\textbf{Case 3:} $a_i=R$ and $b_j=L$. This is identical to case 2, except now a crossing corresponds to $k_1'$ coming in underneath $k_2'$, which we do not count.

\noindent Theorem \ref{thm:mainthm} follows.

\section{Comparison to Kennedy's formula}\label{sec:kennedy}

In \cite{Kennedy94}, Kennedy provides another computation of $\link(W_1, W_2)$, where $W_1$ and $W_2$ are two Lorenz words (that are not equivalent under cyclic shift).

\begin{definition}
	Let $\sigma$ be a permutation of $\{1,2,\ldots,n\}$. Define the crossing count of $\sigma$ to be $\sum_{i=1}^n|\sigma(i)-i|$.
\end{definition}

For $i=1,2$, let $W_i$ have period $n_i$, and write out the following words in order (where $s$ represents a single shift):
\[W_1, s(W_1), s^2(W_1), \ldots, s^{n_1-1}(W_1), W_2, s(W_2), \ldots, s^{n_2-1}(W_2).\]
Let $\sigma_1$ denote the permutation that alphabetizes the first $n_1$ words, let $\sigma_2$ denote the permutation that alphabetizes the second $n_2$ words, and let $\sigma_3$ denote the permutation that alphabetizes all $n_1+n_2$ words.

\begin{theorem}[Theorem 1 of \cite{Kennedy94}]\label{thm:kennedy}
	The linking number of the Lorenz knots with words $W_1$ and $W_2$ is
	\[\dfrac{1}{4}\left(C(\sigma_1)+C(\sigma_2)-C(\sigma_3)\right).\]
\end{theorem}

For example, consider our example of $W_1=RRLLRL$ and $W_2=LR$, as examined in Figures \ref{fig:twoknots} and \ref{fig:topint}. The alphabetizations are in Table \ref{table:alphabet}.

\begin{table}[hbt]
	\centering
	\caption{Alphabetization of $RRLLRL$ and $LR$.}\label{table:alphabet}
	\begin{tabular}{|c|c|c|c|}
		\hline
		$n$ & Word     & Joint order & Individual order        \\ \hline
		1   & $RRLLRL$ & 8 & 6 \\ \hline
		2   & $RLLRLR$ & 6 & 4 \\ \hline
		3   & $LLRLRR$ & 1 & 1 \\ \hline
		4   & $LRLRRL$ & 3 & 2 \\ \hline
		5   & $RLRRLL$ & 7 & 5 \\ \hline
		6   & $LRRLLR$ & 4 & 3 \\ \hline\hline
		7   & $LR$     & 2 & 1 \\ \hline
		8   & $RL$     & 5 & 2 \\ \hline
	\end{tabular}
\end{table}
Therefore $\sigma_1=(6,4,1,2,5,3)$, $\sigma_2=(1,2)$, and $\sigma_3=(8,6,1,3,7,4,2,5)$. This gives
\begin{align*}
	C(\sigma_1)= |6-1|+|4-2|+|1-3|+|2-4|+|5-5|+|3-6|= & 14\\
	C(\sigma_2)= |1-1|+|2-2|= & 0\\
	C(\sigma_3)= |8-1|+|6-2|+|1-3|+|3-4|+|7-5|+|4-6|+|2-7|+|5-8|= & 26,
\end{align*}
hence
\[\link(W_1, W_2)=\dfrac{1}{4}\left(14+0-26\right)=-3,\]
which agrees with our computation. A direct combinatorial proof that the counts in Theorems \ref{thm:kennedy} and \ref{thm:mainthm} are equal does not appear to be obvious.

\section{Comparison to Duke-Imamo\={g}lu-T\'{o}th's symmetrized linking number}\label{sec:DIT}

In certain cases, $\link(k_A, k_B)$ can be deduced from $\link(k_A+k_{A^{-1}},k_B+k_{B^{-1}})$, which was computed in \cite{DIT17} and \cite{JR21}.

\begin{definition}
	A primitive hyperbolic matrix $A\in\PSL(2, \RR)$ is called \textit{reciprocal} if $A$ is conjugate to $A^{-1}$.
\end{definition}

The matrix $A$ is reciprocal if and only if the corresponding quadratic form $q_A$ is reciprocal, i.e. $q_A$ is $\PSL(2, \ZZ)$-equivalent to $-q_A$. Furthermore, if $W$ is the Lorenz word associated to $A$, then $A$ is reciprocal if and only if when you write $W$ backwards and swap $L$'s and $R$'s, you end up with a cyclic shift of $W$. For example, $A=\sm{3}{5}{7}{12}$ (from the example in Figure \ref{fig:twoknots}) is reciprocal, since it corresponds to the Lorenz word $W=RRLLRL$. Writing this word backwards and swapping $R$'s and $L$'s gives $RLRRLL$, which is a cyclic shift of $W$ by 4 places.

Using Theorem \ref{thm:mainthm} (see Proposition 3.8 of \cite{JR21} and the surrounding commentary for more details) we can prove the following result.

\begin{proposition}\label{prop:reciprocal}
	Let $A$ or $B$ be reciprocal. Then
	\[\link(k_A, k_B)=\dfrac{1}{4}\link(k_A+k_{A^{-1}},k_B+k_{B^{-1}}).\]
\end{proposition}

In particular, if either $A$ or $B$ is reciprocal, then $\link(k_A, k_B)$ can be deduced from the symmetrized linking number. If both $A$ and $B$ are not reciprocal, then this is no longer true, and we obtain new information.

For example, consider the non-reciprocal matrices
\[A=\lm{3}{2}{1}{1},\quad B=\lm{14}{9}{3}{2},\quad C=\lm{19}{7}{8}{3}, \quad D=\lm{28}{9}{3}{1},\]
which correspond to the Lorenz words
\[W_A=LLR,\quad W_B=LLLLRLR,\quad W_C=LLRRLRR, \quad W_D=LLLLLLLLLRRR.\]

Table \ref{table:reciprocal} gives the corresponding linking numbers, and it can be observed that Proposition \ref{prop:reciprocal} does not necessarily hold when neither matrix is reciprocal.

\begin{table}[hbt]
	\centering
	\caption{Linking number comparison.}\label{table:reciprocal}
	\begin{tabular}{|c|c|c|c|}
		\hline
		$X,Y$ & $\link(k_X, k_Y)$ & $\link(k_X, k_{Y^{-1}})$ & $\link(k_X+k_{X^{-1}}, k_Y+k_{Y^{-1}})$ \\ \hline
		$A,B$ & 4 & 2 & 12 \\ \hline
		$A,C$ & 3 & 4 & 14 \\ \hline
		$A,D$ & 3 & 3 & 12 \\ \hline
		$B,C$ & 6 & 8 & 28 \\ \hline
		$B,D$ & 7 & 5 & 24 \\ \hline
		$C,D$ & 7 & 7 & 28 \\ \hline
	\end{tabular}
\end{table}

\bibliographystyle{alpha}
\bibliography{../references}
\end{document}